\documentclass[a4paper,11pt]{article}

\usepackage[utf8]{inputenc}
\usepackage{authblk}
\usepackage{amsmath}
\usepackage{amssymb}
\usepackage{amsfonts}
\usepackage{amsthm}
\usepackage{mathrsfs} 
\usepackage{bm}
\usepackage{bbm}
\usepackage{tikz}
\usepackage{centernot}
\usepackage{hyperref}
\usepackage{enumitem}
\usepackage{soul}
\usepackage[ruled,vlined]{algorithm2e}
\usepackage[left=1in, right=1in, top=1.1in, bottom=1.2in]{geometry}

\DeclareMathOperator{\Bin}{Bin}

\newcommand{\E}{\Bbb{E}}

\renewcommand{\P}{\Bbb{P}}

\newcommand{\EE}{\mathcal{E}}

\newcommand{\ep}{\epsilon}

\def\rddots#1{\cdot^{\cdot^{\cdot^{#1}}}}

\hypersetup{
    colorlinks=true,
    linkcolor=blue,
    filecolor=magenta,
    urlcolor=blue,
    citecolor=blue,
    backref=true
}

\makeatletter
\newtheorem*{rep@theorem}{\rep@title}
\newcommand{\newreptheorem}[2]{%
\newenvironment{rep#1}[1]{%
 \def\rep@title{#2 \ref{##1}}%
 \begin{rep@theorem}}%
 {\end{rep@theorem}}}
\makeatother

\numberwithin{equation}{section}

\newtheorem{thm}{Theorem}
\newreptheorem{thm}{Theorem}

\newtheorem{result}{Result}[section]
\newtheorem{lem}[result]{Lemma}
\newtheorem{prp}[result]{Proposition}

\newtheorem{clm}[result]{Claim}

\theoremstyle{definition}
\newtheorem{rmk}[result]{Remark}
\newtheorem*{defn}{Definition}

\theoremstyle{remark}

\newcommand{\hide}[1]{}
\newcommand{\edit}[1]{}

\newcommand{\rough}[1]{}
\definecolor{darkgreen}{RGB}{75,150,75}
\newcommand{\review}[1]{}

\newcommand{\hides}[1]{}
\newcommand{\pub}[1]{}

\title{\vspace{-0.9cm}Induced Ramsey problems for trees and graphs with bounded treewidth}
\author{Zach Hunter\thanks{ETH Zurich, e-mail: \textbf{\{zach.hunter, benjamin.sudakov\}@math.ethz.ch}. Research supported in part by SNSF grant 200021-228014.}, Benny Sudakov$^*$}

\date{}

\begin{document}
\maketitle
\begin{abstract}
    The induced $q$-color size-Ramsey number $\hat{r}_{\text{ind}}(H;q)$ of a graph $H$ is the minimal number of edges a host graph $G$ can have so that every $q$-edge-coloring of $G$ contains a monochromatic copy of $H$ which is an induced subgraph of $G$. A natural question, which in the non-induced case has a very long history, asks which families of graphs $H$ have induced Ramsey numbers that are linear in $|H|$. We prove that for every $k,w,q$, if $H$ is an $n$-vertex graph with maximum degree $k$ and treewidth at most $w$, then $\hat{r}_{\text{ind}}(H;q) = O_{k,w,q}(n)$. 
    This extends several old and recent results in Ramsey theory. Our proof is quite simple and relies upon a novel reduction argument.
\end{abstract}

\section{Introduction}

A celebrated theorem of Ramsey says that for every graph $H$, and every integer $q$, there exists $N$ such that every $q$-edge-coloring of the complete graph $K_N$, contains a monochromatic copy of $H$. We write $r(H;q)$ to denote the smallest integer $N$ with this property, this is the $q$-color Ramsey number of $H$. When $q=2$, we often suppress this, writing $r(H)$ instead of $r(H;2)$. Determining or estimating Ramsey numbers is a central question in Combinatorics, which was extensively studied in the last seventy years. We refer the reader to the book \cite{GRS} or the survey \cite{CFS} for further literature.

More generally, one can consider a host graph $G$ different from the complete graph and say that $G \to (H)_q$ if for every $q$-coloring of the edges of $G$, we can find a monochromatic copy of $H$. The Ramsey numbers ask for the minimal number of vertices a host graph $G$ can have, in which case $G$ should clearly be complete. But one can also consider having ``sparser'' host graphs $G$, which minimize other parameters. This led Erd\H{o}s, Faudree, Rousseau, and Schelp in 1978 to the natural question of the minimum number of edges \cite{EFRS} that a Ramsey graph can have. They define $\hat{r}(H;q):= \min\{|E(G)|: G\to (H)_q\}$ to be the $q$-color size-Ramsey number (of $H$). 

One of the extensively studied questions, which was asked by Erd\H{o}s \cite{erdos}, is to understand which classes of graphs have size-Ramsey numbers that are linear in their number of edges. A classic result of Beck \cite{beck} is that the $n$-vertex path $P_n$ has linear size-Ramsey numbers. Later on, Friedman and Pippenger proved that $n$-vertex trees $T$ with maximum degree $k$ also satisfies $\hat{r}(T;q) = O_{q,k}(n)$ \cite{FP}. Building upon this, Haxell and Kohayakawa proved the bound $\hat{r}(T;q) = O_q(k n)$ \cite{HK}. This is tight in general, since as observed by Beck \cite{beck2} it is easy to construct trees $T$ where $\hat{r}(T)\ge \Omega(k n)$.

Paths and trees are very special cases of graphs with bounded treewidth. Recall that a graph is $G$ is \textit{chordal} if every induced cycle is a triangle, and that $H$ has treewidth at most $w$ if $H\subset G$ for some chordal graph $G$ without a clique of size $w+2$. It was recently proved by Berger, Kohayakawa, Maeseka, Martins, Mendon\c{c}a, Mota, and Parczyk that if $H$ is an $n$-vertex graph with maximum degree at most $k$ and treewidth at most $w$, then $\hat{r}(H;q) =O_{q,k,w}(n)$ \cite{BKMMMMP} (see also \cite{KLWY}, which handles the case of $q=2$). Furthermore Dragani\'c, Kaufmann, Munh\'a Correia, Petrova, and Steiner \cite{DKMPS}
has obtained the nearly linear bound $\hat{r}(H)\le O_k(nw\log n)$ on the size-Ramsey numbers of $H$ with maximum degree $k$ and treewidth $w$ growing with $n$.

In this paper, we will focus on the induced analogues of the above quantities. We say $G\to_{\text{ind}} (H)_q$ if for every $q$-coloring of $E(G)$, there is some monochromatic copy of $H$ which is an induced subgraph of $G$. We define $r_{\text{ind}}(H;q)$ and $\hat{r}_{\text{ind}}(H;q)$ in the same way as the non-induced versions, being the minimum number of vertices/edges the host graph $G$ must have to satisfy $G\to_{\text{ind}} (H)_q$. The existence of these numbers is non-trivial and is an important extension of Ramsey's theorem. It was originally proven independently by Deuber \cite{deuber}, Erd\H{o}s, Hajnal, and Posa \cite{EHP}, and R\"{o}dl \cite{rodl}. Naturally, it is much harder to get upper bounds for these induced variants. For example, a short and simple argument shows that $r(H)\le 4^n$ for any $n$-vertex graph $H$. On the other hand, the best upper bound in the induced setting is only $r_{\text{ind}}(H)\le n^{O(n)}$ (see \cite{CFS1}), and it is a famous open question of Erd\H{o}s whether there is an exponential bound $2^{O(n)}$ for this problem. 

Another natural question (see e.g., \cite{BDS}), that had a gap between bounds for induced and non-induced settings, was bounded degree trees.
Although, as we mentioned above, it was well known that the size-Ramsey number of such trees are linear, the induced counterpart was not known.
Very recently, Gir\~ao and Hurley \cite{GH} proved that for any tree $T$ with maximum degree $k$, $\hat{r}_{\text{ind}}(T;q) = O_q(k^2n)$. This is a very nice result that in particular shows how to find induced trees in sparse expander graphs. The main contribution of this note is a shorter, alternative proof of a more general result (albeit with worse quantitative bounds).

\begin{thm}\label{thm:bonus} Let $H$ be an $n$-vertex graph with maximum degree $k$ and treewidth at most $w$. Then 
    \[\hat{r}_{\text{ind}}(H;q)  = O_{k,w,q}(n).\]
\end{thm}
\noindent In particular, this theorem extends a celebrated result of Haxell, Kohayakawa, and Luczak \cite{HKL} that the induced size-Ramsey numbers of cycles are linear (since cycles have treewidth $2$).

While \cite{GH} works by extending Friedman-Pippenger-type techniques to the induced setting, our approach is quite different. We use a novel reduction (Theorem~\ref{thm:reduction general}, discussed in the next subsection) which converts constructions for the size-Ramsey numbers $\hat{r}(H;q)$ into constructions for $r_{\text{ind}}(H;q)$ and $\hat{r}_{\text{ind}}(H;q)$. Then, we obtain Theorem~\ref{thm:bonus} by applying the result of \cite{BKMMMMP} as a blackbox, which implicitly builds on the standard \textit{non-induced} Friedman-Pippenger embedding method.

Due to the usage of (weak) regularity-based arguments, the absolute constants from Theorem~\ref{thm:bonus} are quite large. Therefore for trees, we shall also present a more efficient version of our argument, which we believe may be of independent interest. We first need a key definition, that will be used throughout the paper.
\begin{defn}
    Given a graph $G$, we say $G'$ is an \textit{$s$-blowup} of $G$, if there is a homomorphism $\phi:V(G')\to V(G)$ so that $uv\in E(G')$ implies $\phi(u)\phi(v)\in E(G)$ and furthermore $|\phi^{-1}(v)| \le s$ for each vertex $v\in V(G)$. In particular, we do not have edges inside the subsets $\phi^{-1}(v)$.
\end{defn}

\begin{thm}\label{thm:main}
    Let $T$ be an $n$-vertex tree with maximum degree $k$. Then 
    \[r_{\text{ind}}(T;q)\le 
    \hat{r}_{\text{ind}}(T;q)\le (kq)^{Cq^4k^3}n\]for some absolute constant $C$. Furthermore, if $H$ is a $w$-blowup of $T$, then 
    \[\hat{r}_{\text{ind}}(H;q)\le (kq)^{Cq^4k^3 w}n.\] 
\end{thm}
\noindent The latter bound gives the correct dependence on $w$. Indeed, the complete $w$-blowup $H$ of $P_n$ contains $\Omega(n)$ vertex disjoint copies of $K_{w,w}$. Meanwhile (for $w\ge 11$), using the local lemma it is easy to show that any graph with maximum degree at most $2^{w/2}$ has a $2$-edge coloring without any monochromatic copy of $K_{w,w}$ (see Lemma \ref{path-blowup}). Thus if $G\to H$, $G$ must have $\Omega(n)$ vertices of degree $\ge 2^{w/2}$, and so $e(G)\ge \Omega(2^{w/2}n)$. We note that in the non-induced setting an upper bound of $\hat{r}(H;q)\le (kq)^{q^{O(qk)}w}n$ was obtained by Jiang, Milans, and West \cite[Theorem~5.3]{JMW}.

\subsection{Idea of proof}\label{idea of proof} A large amount of work on the size-Ramsey numbers and induced Ramsey numbers (including \cite{beck,FS,DGK,GH}) have considered the host graph $G$ to be a random graph $G(N,p)$ for appropriately chosen $N,p$ and used the generic pseudorandom properties it satisfies. However, we shall instead argue by using a more carefully constructed host graph (taking a ``gadget-based'' approach).

Our construction will start with a non-induced Ramsey host graph $G$ of $H$, where $\Delta(G)$ is bounded, and then consider an appropriate ``pseudorandom blowup'' of $G$. More formally, to prove Theorem~\ref{thm:bonus}, we rely on the following.

\begin{thm}\label{thm:reduction general}
    Fix $k,\Delta,q$, there exists some $s = s(k,\Delta,q)$ so that the following holds. 
    Let $H,G$ be graphs with $G\to (H)_q$. Suppose $\Delta(H)\le k$ and $\Delta(G)\le \Delta$.
    Then, there is an $s$-blowup of $G$, $G'$, so that $G'\to_{\text{ind}} (H)_q$. 
\end{thm}

\noindent
To prove Theorem~\ref{thm:main}, we will use a bipartite-analogue of this result (Theorem~\ref{thm:reduction bipartite}), which has stronger bounds (and adds a new parameter $w$ to get the second part of Theorem~\ref{thm:main}). Theorems~\ref{thm:reduction general} and \ref{thm:reduction bipartite} are the two main results of our work. 

Next, we describe our construction of $G'$ in a bit more detail. We start with some bounded degree host graph $G$ satisfying $G\to (H)_q$. Then, we take some ``pseudorandom bipartite gadget $\Gamma$'' with vertex sets of size $s$ on both sides. We define our new host graph $G'$ by replacing each $v\in V(G)$ by a set $X_v$ of $s$ vertices and for each $uv\in E(G)$ adding a copy of $\Gamma$ between $X_u$ and $X_v$. In this paper our gadgets are properly chosen random bipartite graphs, but we can also choose $\Gamma$ to be a dense spectral expander, for which there are various explicit constructions (e.g., Example~5 from the survey \cite[Section~3]{KS}). To show that $G'$ is induced-Ramsey for $H$, we employ the following high-level strategy.
\begin{enumerate}
    \item First, fix some $q$-coloring $C'$ of $E(G')$.
    \item We then apply a ``cleaning procedure'', producing a $q$-coloring $C$ of $E(G)$, and subsets $(X_v^*)_{v\in V(G)}$, so that for any edge $e=uv\in E(G)$, the $C(e)$-monochromatic subgraph of $G'[X_u^*,X_v^*]$ is appropriately ``robust'' (for Theorem~\ref{thm:reduction general}, we shall require said graph is lower-regular).
    \item By assumption, since $G\to (H)_q$, we can find a monochromatic copy of $H$ inside $G$.
    \item Finally, using this monochromatic copy of $H$, we run an embedding procedure to find a monochromatic induced copy of $H$ inside $G'$. Here we use the pseudorandomness assumptions about $G'$ together with the fact that $G$ having bounded degree, to get all our desired non-edges. Meanwhile, the robustness properties ensured by Step~2 are used so that we get our desired monochromatic edges. 
\end{enumerate}
\noindent In the above, it will be important that $s$ is sufficiently large. Indeed, to run our ``induced embedding procedure'' (Step 4), we will need a certain lower bound on the sizes of our sets $X_v^*$ produced by (Step 2). Meanwhile, in the ``cleaning procedure'' (Step 2), our sets will be forced to shrink by some constant factor. Both the lower bound and this constant factor will end up depending on $k,\Delta,q$.

Given the outlined framework, a proof of Theorem~\ref{thm:bonus} can be phrased quite concisely. First, we take an appropriate pseudorandom blowup (as described above) of the Ramsey graph $G$, that is provided to us by the work of \cite{BKMMMMP}. For the cleaning procedure needed in Step~2, we can use Proposition~\ref{regularity cleaning} from \cite{CNT}. 
Hence, the only new ingredient we require is a simple embedding result for Step~4, which we prove in Proposition~\ref{general greedy embedding}. Putting these three pieces together immediately gives Theorem~\ref{thm:bonus}.
\section{Preliminaries}
\subsection{Notation}

The graph theoretic notation is mostly standard, although we recall a few specific concepts here. We denote an edge $e$ between two vertices $u$ and $v$ by $uv$. Given a graph $G$, we write $N_G(x)$ to denote the neighborhood of $x$. For a vertex $x\in V(G)$ and $S\subset V(G)$, we write $d_S(x) := |N_G(x)\cap S|$ to count the number of neighbors $x$ has in $S$. Given sets $A,B\subset V(G)$, we write $G[A]$ to denote the subgraph induced by $A$, and $G[A,B]$ to denote the bipartite induced between $A$ and $B$.

\subsection{Lemmata}
In this subsection we collect some technical lemmas which we will need in our proofs. We start by describing pseudorandom bipartite graphs, which we will use for blowups.

\begin{defn}
    We say a bipartite graph $G=(X,Y,E)$ is {\it $(L,p)$-regular} if: for any $X'\subset X$ of size $|X'|\ge L$, there are at most $L$ vertices $y \in Y$ such that  
 $d_{X'}(y)<\frac{1}{2}p|X'|$ or $d_{X'}(y)>2p|X'|$, and similarly for any $Y'\subset Y$ with $|Y'|\ge L$ there are at most $L$ vertices $x \in X$ satisfying
 $d_{Y'}(x)<\frac{1}{2}p|Y'|$ or $d_{Y'}(x)> 2p|Y'|$.
\end{defn}
\noindent We also say that $G = (X,Y,E)$ is $(L,p)$-lower-regular if for any $X'\subset X$ of size $|X'|\ge L$, we have $|\{y\in Y: d_{X'}(y)<\frac{1}{2}p|X'|\}|<L$ (and the similar inequality holds for $Y'\subset Y$ with $|Y'|\ge L$).

A fully standard application of the probabalistic method yields the following claim, which we will use to construct gadgets.
\begin{prp}\label{regular}
      Consider $p\in (0,1)$ and $a,b$. There exists a bipartite graph $\Gamma= (A,B,E)$ which is $(\frac{48}{p}\ln(a+b),p)$-regular. 
\end{prp}

To prove this proposition we use the following lemma.

\begin{lem}\label{regular G}
    Consider $p,\ep \in (0,1)$ and $t\ge 1$. Assume $n\le \exp\left(\frac{t\ep^2 p}{6}\right) $. Then there exists an $n$-vertex graph $G$, so that for any two disjoint sets $X',Y'\subset V(G)$ of size $t$, \[|e_G(X',Y')-pt^2| \le \ep p t^2.\] 
    \begin{proof}
        We sample $G\sim G(n,p)$. Let $\mathbf{Z}$ count the number of unordered ``bad'' pairs of disjoint $X',Y'\in \binom{V(G)}{t} $ with 
        \[|e_G(X',Y')-pt^2|> \ep p t^2.\]
For any fixed choice of $X',Y'$, we estimate the probability that they are bad using a standard Chernoff
bound (see \cite[Corollary~A.1.14]{AS}). It says that for any $\delta\in (0,1)$ and binomial random variable $X\sim \Bin(n,p)$, 
\begin{eqnarray}
    \label{chernoff}
    \P(|X-np|>\delta np) \le 2 e^{-\delta^2 np/3}.
\end{eqnarray}
 Since the number of edges between $X',Y'$ is distributed like $\Bin(t^2,p)$, we can bound
the probability that they are bad by $2e^{-\ep^2p t^2/3}$. So, considering all possible unordered pairs, we have that
        \[\E[\mathbf{Z}] \le \frac{1}{2}\binom{n}{t}\binom{n-t}{t} \cdot  2e^{-\ep^2p t^2/3}< n^{2t} 
        e^{-\ep^2p t^2/3}.\] Using our assumptions on $n$, it is easy to check that $\E[\mathbf{Z}]<1$, whence there must be some outcome of $G$ with the desired properties.
    \end{proof}
\end{lem}

    \begin{proof}[Proof of Proposition~\ref{regular}]
        Write $n:= a+b, \ep:=1/2,p:=p, t:= \frac{6}{\ep^2 p}\ln n = \frac{24}{p}\ln n$. We can apply Lemma~\ref{regular G} with these parameters to get some appropriate $n$-vertex graph $G$. We will take $\Gamma = G[A,B]$ where $A,B$ are any two disjoint sets of size $a,b$ respectively, and argue this is $(2t,p)$-regular (as desired). 
        
Indeed, suppose there was some set $X'\subset X$ of size at least $2t$ so that there are at least $2t$ vertices $y\in Y$ satisfying $d_{X'}(y)<\frac{1}{2}p|X'|$ or $2p|X'|<d_{X'}(y)$. By pigeonhole principle, we can find some $Y'$ (disjoint from $X'$) of size $t$ so that either $d_{X'}(y)> 2p|X'|>(1+\ep)p|X'|$ for all $y\in Y'$ or $d_{X'}(y)
<(1-\ep)p|X'|$ for all $y\in Y'$.
Taking a random $X''\subset X'$ of size exactly $t$, we either have $\E[e(G[X'',Y'])]>(1+\ep)p t^2$ or $\E[e(G[X'',Y'])]<(1-\ep)p t^2 $, contradicting in either case the properties of $G$. Arguing symmetrically for sets $Y'\subset Y$ establishes that $\Gamma$ is $(2t,p)$-regular, as desired.
    \end{proof}

Next, we prove the bound from the introduction which we use to show the tightness of Theorem \ref{thm:main} for blow-ups of trees 
\begin{lem}
   \label{path-blowup}
Fix $w\ge 11$. Let $G$ be a graph with maximum degree at most $2^{w/2}$. Then $G$ has a $2$-edge coloring with no monochromatic $K_{w,w}$.
\begin{proof}
         To prove this statement we use the Lov\'asz Local Lemma (see, e.g., \cite[Lemma~5.1.1]{AS}), which says that if $(E_i)_{i\in I}$ is a collection of events so that $E_i$ is mutually independent from all but at most $D$ other events $E_j$ and 
        \begin{eqnarray}
        \label{local-lemma}
            \P(E_i) \le \frac{1}{e(D+1)},
        \end{eqnarray}
  then with positive probability none of the events $E_i$ occur.      
        
Consider a uniformly random 2-edge coloring of $G$. For each $H\subset G$ isomorphic to $K_{w,w}$, let $E_H$ be the event that $H$ is monochromatic. Clearly, $\P(E_H) = p:= 2^{1-w^2}$ for each $H$. Moreover, two events $E_H, E_{H'}$ can be dependent only if $H,H'$ share an edge. We can now use that $G$ has bounded degree to control dependencies. Indeed, for any $e=uv\in E(G)$, the number of copies $H$ containing $e$ is at most $\binom{d(u)}{w-1} \binom{d(v)}{w-1} \le 2^{w^2-w}\le \frac{1}{12 w^2}2^{w^2}$ (assuming $w\ge 11$). Since every copy of $K_{w,w}$ has $w^2$ edges, each event is independent from all but at most $D \leq \frac{1}{12}2^{w^2}$ other events. As $e(D+1)p<1$ we can use the local lemma to get 
a coloring without monochromatic $K_{w,w}$.
    \end{proof}
\end{lem}

Finally, we need a short lemma whose proof is based on \textit{dependent random choice}, which is a powerful tool in extremal combinatorics (see \cite{FS2} for a discussion of many applications). Given a sufficiently dense bipartite graph $\Gamma=(X,Y,E)$, this method allows one to pass to some subgraph $\Gamma'= (X,Y',E)$ so that any $k$ vertices in $Y'$ have many common neighbors in $X$. 

However, for our applications, we need a slightly more technical statement. Specifically, given multiple subsets $Y_1,\dots,Y_t\subset Y$, and appropriate assumptions on the subgraphs $\Gamma[X,Y_i]$, we want to find an outcome of $Y'$ where $Y'\cap Y_i$ is simultaneously non-empty for each $i=1,\dots,t$. This forces us to be more careful and use a simple but slightly different variant of the usual dependent random choice argument.

\begin{lem}\label{basic convexity}
    Let $h\ge 1$, $p\in (0,1)$ and let $\Gamma =(X,Y,E)$ be a bipartite graph with at least $p |X||Y|$ edges. If we sample $x_1,\dots,x_h \in X$ uniformly at random with repetitions, then
    \[\P\big(|\cap_{i=1}^h N(x_i)|\ge \frac{p^h}{2}|Y|\big)>\frac{p^h}{2}|Y|. \]
\end{lem}
\begin{proof}
    Consider the random variable $\mathbf{Z}:=|\cap_{i=1}^h N(x_i)|$. 
    Using Jensen's inequality, we have $$\E[\mathbf{Z}] = \sum_{y\in Y} (d(y)/|X|)^h \ge 
    |Y|\cdot \left(\frac{\sum_{y\in Y} d(y)}{|X||Y|}\right)^h\geq p^h |Y|.$$ 
   Since $\mathbf{Z}\le |Y|$ always holds, we have $$p^h|Y|\le \E[\mathbf{Z}]\le |Y|\cdot \P(\mathbf{Z}\ge p^h |Y|/2)+p^h|Y|/2.$$ Rearranging gives $\P(|N(x_1)\cap \dots \cap N(x_h)|\ge \frac{p^h}{2}|Y|)\geq p^h/2$, as required.
\end{proof}

\begin{prp}\label{simultaneous DRC}
    Fix integers $\ell,h,r$. Let $\Gamma =(X,Y,E)$ be a bipartite graph and let $Y_1,\dots,Y_\ell$ be subsets of $Y$ so that $|N(x)\cap Y_i|\ge p |Y_i|$ for each $x\in X$ and $i\in [\ell]$.
    Suppose that $p^{h\ell}/2 > |Y|^r|X|^{-h/2}$. Then
    there exists subsets $Y_i'\subset Y_i, i\in [\ell]$ such that $|Y_i'|\ge \frac{p^{h\ell}}{2}|Y_i|$ for $i\in [\ell]$ and 
     for any $y_1,\dots,y_r\in \bigcup_i Y_i'$, we have that $|N(y_1)\cap \dots N(y_r)|\ge \sqrt{|X|}$.
\end{prp}
\begin{proof}
    Sample $x_1,\dots,x_h \in X$ uniformly at random with repetitions, and set $Y' := \bigcap_{i=1}^h N(x_i)$. Let $\EE_{\text{good}}$ be the event $|Y_i\cap Y'|\ge \frac{p^{h\ell}}{2}|Y_i|$ for all $i\in [\ell]$, and $\EE_{\text{bad}}$ be the event that $|\cap_{i=1}^r N(y_i)|<|X|^{1/2}$ for some choice of $y_1,\dots,y_r\in Y'$. 
    We prove that $\P(\EE_{\text{good}})>\P(\EE_{\text{bad}})$, which implies the existence of an outcome of $Y'$ satisfying only our good event. Taking $Y_i':= |Y_i\cap Y'|$ for $i\in [\ell]$ gives the desired result. 

    First, we note that $\P(\EE_{\text{bad}})\le |Y|^r|X|^{-h/2}$. Indeed, for any fixed choice of $y_1,\dots,y_r\in Y$, we have that $\P(\{ y_1,\dots,y_r\}\subset Y') = \left(|\cap_{i=1}^r N(y_i)|/|X|\right)^h$. In particular, when $|\cap_{i=1}^r N(y_i)|<|X|^{1/2}$,  this probability is at most $|X|^{-h/2}$. Taking a union bound over all such tuples gives $\P(\EE_{\text{bad}})\le |Y|^r|X|^{-h/2}$.

    We now estimate the probability of the good event. Define an auxiliary bipartite graph $\Tilde{\Gamma}$ with parts $X$ and $\Tilde{Y}:=Y_1\times \dots \times Y_\ell$. For $x\in X$, we say $x\sim_{\Tilde{\Gamma}} (y_1,\dots,y_\ell) $ if $\{y_1,\dots,y_\ell\}\subset N_\Gamma(x)$. 
    
    Let $\Tilde{Y}':= N_{\Tilde{\Gamma}}(x_1)\cap\dots \cap N_{\Tilde{\Gamma}}(x_h)$, where these $x_i$ come from the same random tuple which determined $Y'$. Noting that $|\Tilde{Y}'| = \prod_{i=1}^\ell |Y_i\cap Y'|$, we see that $$\frac{|\Tilde{Y}'|}{|\Tilde{Y}|}=\prod_{i=1}^\ell \frac{|Y_i\cap Y'|}{|Y_i|} \le \min_{i\in [\ell]}\left\{\frac{|Y_i\cap Y'|}{|Y_i|}\right\}.$$ Thus the probability $\EE_{\text{good}}$ holds is at least $\P(|\Tilde{Y}'|\ge \frac{p^{h\ell}}{2}|\Tilde{Y}|)$. 

    By the minimum degree assumptions, $d_{\Tilde{\Gamma}}(x) = \prod_{i=1}^\ell |N_\Gamma(x)\cap Y_i| \ge \prod_{i=1}^\ell p |Y_i| = p^\ell |\Tilde{Y}|$ for every $x\in X$. It follows that $\Tilde{\Gamma}$ has density at least $p^\ell$. Applying Lemma~\ref{basic convexity}, we have that $\P(|\Tilde{Y}'|\ge \frac{p^{h\ell}}{2}|\Tilde{Y}|)\ge \frac{p^{h\ell}}{2}$. Using this, together with the assumption $p^{h\ell}/2 > |Y|^r|X|^{-h/2}$, gives $\P(\EE_{\text{good}})>\P(\EE_{\text{bad}})$, completing the proof. 
\end{proof}

\section{Induced embedding}\label{embedding section}
In this section, we prove two ``induced embedding results'' for pseudorandom blowups of graphs. The idea behind the first result is to employ a greedy embedding procedure, using pseudorandomness (see Property~2 below) to control our non-edges. 
\begin{prp}[Induced embedding with regularity]\label{general greedy embedding}
    Let $H\subset G$ be graphs, with $\Delta(H) \le k$ and $\Delta(G) \le \Delta$. Let $s^*,L,L',p,\rho$ be constants so that
    \begin{equation}\label{s-star large}
        s^*(\rho/2)^k (1-2p)^\Delta > \Delta L+kL'. 
    \end{equation}

    Now suppose there are graphs $H^* \subset G^*$, along with a homomorphism $\phi:G^* \to G$, so that writing $X_v := \phi^{-1}(v)$ for $v\in V(G)$, we have:
    \begin{enumerate}
        \item $|X_v| \ge s^*$ for $v\in V(G)$;
        \item $G^*[X_u,X_v]$ is $(L,p)$-regular for $uv\in E(G)$;
        \item $H^*[X_u,X_v]$ is $(L',\rho)$-lower-regular for $uv\in H$.
    \end{enumerate}
    Then, we can find a set of vertices $W$ so that $H^*[W] \cong H \cong G^*[W] $.
\end{prp}
\begin{proof}[Proof of Proposition~\ref{general greedy embedding}]
Write $h:= |V(H)|$. Let $v_1,\dots,v_h$ be any ordering of $V(H)$. For $i\in [h]$, let $J_i := \{j<i: v_j\in N_H(v_i)\},\overline{J}_i := \{j'<i : v_{j'}\in N_G(v_i) \setminus N_H(v_i)\}$. For $t= 0,\dots,h$, let $H^{(t)} := H[\{v_1,\dots,v_t\}]$ denote the subgraph of $H$ induced by the first $t$ vertices of our ordering.

We embed vertices $v_i$ one by one using the following iterative procedure. For every $0\leq t \leq h-1$, at the beginning of stage $t+1$ we have already chosen vertices $x_1,\dots,x_t$, and for $i>t$ have sets 
\[X_i^{(t)} := X_{v_i}  \cap \Big(\bigcap_{j\in [t] \cap J_i } N_{H^*}(x_j)\Big) \setminus \Big(\bigcup_{j'\in [t] \cap \overline{J}_i } N_{G^*}(x_{j'}) \Big)\] which satisfy the following conditions.
    \begin{itemize}
        \item The map $x_i \mapsto v_i$ is an isomorphism into $H^{(t)}$ for both $H^*$ and $G^*$,
        \item and $|X_i^{(t)}|\ge s^* (\rho/2)^{|[t]\cap J_i|} (1-2p)^{|[t] \cap \overline{J}_i|}$ for each $i>t$.
    \end{itemize}
Note that for $t = 0$, this is vacuously satisfied, since $X_i^{(0)} = X_{v_i}$ and has size at least $s^*$ by Property~1. Now, assuming that our assumptions hold at the end of stage $t-1$, we show how to pick $x_t\in X_t^{(t-1)}$ appropriately. 

By definition, we have that every $x_t \in X_t^{(t-1)}$ satisfies $x_t \in N_{H^*}(x_j) $ for $j\in J_t$, and $x_t\not\in N_{G^*}(x_{j'})$ for $j'\in \overline{J}_t$. We further have $x_t \not \in N_{G^*}(x_{j''})$ for $j'' \in [t-1] \setminus (J_t\cup \overline{J}_t) $, since $v_{j''}$ is not adjacent to $v_t$ in $G$ and $G^*$ is homomorphic to $G$. This shows that our first bullet holds for any choice of $x_t \in X_t^{(t-1)}$.

To ensure that there is a choice of $x_t$ which satisfies our second bullet
we use our regularity assumptions. First, by our hypotheses, we have for $i>t-1$ that \begin{equation}\label{blob size}
        |X_i^{(t-1)}| \ge s^* (\rho/2)^{|[t]\cap J_i|} (1-2p)^{|[t] \cap \overline{J}_i|}> \Delta L +kL'\ge \max\{L,L'\}.
    \end{equation} 
This uses the assumption given by Eq.~\ref{s-star large} togther with $|[t-1]\cap J_i|\le d_H(v_i)\le k,|[t-1]\cap \overline{J}_i|\le d_G(v_i)\le \Delta$.
    
Next, let $I := \{i>t : t\in J_i\},\overline{I} := \{i>t: t\in \overline{J}_i\}$. For $i\in I$, set 
$B_i$ to be a set of vertices $x\in X_{v_t}$ which in the graph $H^*$ have less than $\frac{\rho}{2}|X_i^{(t-1)}|$ neighbors in $X_i^{(t-1)}$. Similarly for $i'\in \overline{I}$ define \[B_{i'} := \big\{x\in X_{v_t}: |X_{i'}^{(t-1)}\setminus N_{G^*}(x)|<(1-2p)|X_{i'}^{(t-1)}|\big\}. \]
Recalling Eq.~\ref{blob size}, we have $|X_i^{(t-1)}|\ge L$ for each $i\in I$. Whence, by Property~3 we get that $|B_i|\le L'$. Similarly, by Property~2 we get $|B_{i'}|\le L$ for all $i'\in \overline{I}$. 

Using that $|I|\le d_H(v_t)\le k,|\overline{I}|\le d_G(v_t) \le \Delta$ together with Eq.~\ref{blob size}, we deduce 
$|(\cup_{i\in I}B_i) \cup (\cup_{i'\in \overline{I}}B_{i'})|\le \Delta L+kL' < |X_t^{(t-1)}|$. Therefore, there exists a choice of $x_t \in X_t^{(t-1)}$ not belonging to any of the sets $B_i, B_{i'}$. Picking this $x_t$ gives the desired lower bounds for each $|X_i^{(t)}|, i>t$ and allows us to continue the next embedding iteration.
\end{proof}

Next, we present another induced embedding result. Compared to Proposition~\ref{general greedy embedding}, where we assumed that $H^*[X_u,X_v]$ was lower-regular between parts, we now have a weaker assumption (Property~3 below) that can be obtained using dependent random choice. 
Consequently, we can no longer greedily embed vertices one at a time in our analysis. Instead, we shall utilize the Lov\'asz Local Lemma (see Eq.~\ref{local-lemma}).
Here we will be embedding $w$ vertices into each blob of the host graph, to find induced copies of $w$-blowups $H'$ of $H$. This does not pose any challenges beyond making our notation slightly more cumbersome. Nothing is lost in the proof by assuming $w=1$ and replacing instances of `$(a,j)$', `$(b,i)$' respectively by `$a$', `$b$'. 
\begin{prp}\label{DRC embedding}
    Let $H\subset G$ be bipartite graphs with a common vertex bipartition $(A,B)$; assume that $\Delta(H) =k, \Delta(G) = \Delta$. Finally consider some $w$-blowup $H'$ of $H$.

    Now let $H^*\subset G^*$ be graphs, with a homomorphism $\phi:G^*\to G$, so that writing $X_a := \phi^{-1}(a)$ for $a\in A$ and $Y_b:= \phi^{-1}(b)$ for $b\in B$, we have:
    \begin{enumerate}
        \item $|Y_b| \ge w s^*$ for $b\in B$;
        \item $G^*[X_a,Y_b]$ is $(L,p)$-regular for $ab\in E(G)$;
        \item given any $a\in A$, $b_1,\dots,b_{wk}\in N_H(a)$ and $y_i\in Y_{b_i}$ for $i\in [wk]$, we have that \[(1-2p )^{\Delta w}\big|X_a \cap \bigcap_{i=1}^{wk} N_{H^*}(y_i)\big| \ge wL.\]
    \end{enumerate}
    Then, assuming \begin{equation}\label{local lemma hypothesis}
        w\Delta \frac{L}{s^* }\le \frac{1}{e(w\Delta^2+1)},
    \end{equation} we can find a set of vertices $W$ so that $H^*[W]\cong H' \cong G^*[W]$.
\end{prp}

\begin{proof}
    We may write $V(H') = V(H)\times [w]$ so that $(v,i)\mapsto v$ is a homomorphism from $H'$ to $H$. Next, for $b\in B$, we shall fix $w$ disjoint sets $Y_{b,1}\dots,Y_{b,w}\subset Y_b$ each of size $s^*$ (which is possible by Property 1).

    For $(b,i)\in B\times [w]$, we shall pick a random vertex $y_{b,i}\in Y_{b,i}$ (uniformly at random, and independently). For $(a,j)\in A\times [w]$, we let $T_{a,j} \subset X_a$ be the set of ``valid images'' for $(a,j)$, i.e.,
    \[T_{a,j} := X_a \cap \Bigg(\bigcap_{(b,i)\in N_{H'}(a,j)} N_{H^*}(y_{b,i}) \setminus \bigcup_{\substack{(b',i'): b'\in N_G(a),\\ (b',i')\not\in N_{H'}(a,j)}} N_{G^*}(y_{b',i'})\Bigg).\]

    We let $\EE_{a,j}$ be the ``bad'' event that $|T_{a,j}| <w$. Note that if there is an outcome where none of the $\EE_{a,j}$ hold, then we can greedily pick distinct vertices $x_{a,j}$ for $(a,j)\in A\times[w]$ so that $x_{a,j}\in T_{a,j}$. Indeed, we will pick $w$ vertices $x_{a,1},\dots,x_{a,w}$ inside $X_a$, and after embedding the first $t<w$ of them, we still have $w-t>0$ choices for $x_{a,t+1}$ inside $T_{a,t+1}\setminus \{x_{a,1},\dots,x_{a,t}\}$. The map $(a,j)\mapsto x_{a,j},(b,i)\mapsto y_{b,i}$ shall then produce an induced copy of $H$. This follows from the definition of the sets $T_{a,j}$, and the fact that there are no edges between vertices in $\bigcup_{a\in A} X_a$ and $\bigcup_{b\in B} Y_b$. So it suffices to find such an outcome, which we do using the Lov\'asz Local Lemma.

    Observe that the event $\EE_{a,j}$ is mutually independent from all events
    $\EE_{a',j'}$ for which $N_G(a)\cap N_G(a') = \emptyset$ (in which case $a,a'$ are adjacent in the square graph of $G$, which which has maximum degree at most $\Delta^2$). It follows that our dependency graph has maximum degree at most $D:=w\Delta^2$. So by Eq.~\ref{local-lemma}, we are done if we can prove $\P(\EE_{a,j})\le \frac{1}{e(w \Delta^2+1)}$. 
    
This will follow from Eq.~\ref{local lemma hypothesis} together with the following claim. 
    \begin{clm}\label{local fail bound}
        For $(a,j)\in A\times [w]$, we have that $\P(\EE_{a,j})\le w\Delta\frac{L}{s^*}$.
        \begin{proof}
        
            Write $$T^{(0)} =T_{a,j}^{(0)}:= X_a \cap\Big(\bigcap_{(b,i)\in N_{H'}(a,j)} N_{H^*}(y_{b,i})\Big).$$ Since $|N_{H'}(a,j)|\le w|N_H(a)|\le w\Delta(H)\le wk$ (by assumption, recalling $H'$ is a $w$-blowup of $H$), we have by Property~3 that $(1-2 p)^{w\Delta}|T^{(0)}|\ge wL$.

            We now consider the set $\overline{N}:= \{(b,i)\in B\times [w]: b\in N_G(a)\text{ but }(b,i)\not\in N_{H'}(a,j)\}$ of potential non-neighbors. By assumption, $\ell := |\overline{N}|\le wd_G(a)\le w\Delta$. We fix some arbitrary ordering $(b_1,i_1),\dots,(b_\ell,i_\ell)$ of these vertices. For $t=1,\dots,\ell$, let 
            \[T^{(i)} := T^{(i-1)}\setminus N_{G^*}(y_{b_t,i_t}).\]We wish to show that $\P(|T^{(\ell)}|<w)\le w\Delta\frac{L}{s^*}$. The calculation is similar to Proposition~\ref{general greedy embedding}.
            
            As noted before, we deterministically have $|T^{(0)}|\ge (1-2p)^{-w\Delta} wL \ge (1-2p)^{-\ell} wL$. For $t=1,\dots,\ell$, we now let $\EE^{(t)}$ be the event that $|T^{(t-1)}|\ge L$, but $|T^{(t)}|<(1-2p)|T^{(t-1)}|$. So if none of the events $\EE^{(t)}$ happens, we get $|T^{(\ell)}|\ge (1-2p)^\ell|T^{(0)}|\ge wL\ge w$, as desired. 

            To bound the probability that $\EE^{(t)}$ holds for some $t\in [\ell]$, we show that for every $t$, $\P(\EE^{(t)})< \frac{L}{s^*}$. Indeed, conditioned on $|T^{(t-1)}| \ge L$, Property~2 tells us there are less than $L$ ``bad'' vertices $y\in Y_{b_t,i_t}$ where $|N_{G^*}(y)\cap T^{(t-1)}|>2p|T^{(t-1)}|$. Moreover, in order for the event 
            $\EE^{(t)}$ to hold we neeed that the random vertex $y_{b_t,i_t}$, which is chosen from the set $Y_{b_t,i_t}$ of size $s^*$, is bad. This happens with probability $\P(\EE^{(t)})<\frac{L}{s^*}$. Recalling $\ell \le w\Delta$, a union bound gives $\P\big(\EE^{(t)}\text{ holds for some }t\in [\ell]\big)\le \ell \frac{L}{s^*}\le w\Delta \frac{L}{s^*}$, as desired.            
        \end{proof}
    \end{clm}
\end{proof}

\section{Cleaning results}
In this section, we prove our cleaning results which correspond to Step~2 of our strategy from Subsection~\ref{idea of proof}. Recall that in Section~\ref{embedding section}, our embedding results had three kinds of assumptions; a largeness assumption (Property~1), a sparsity assumption (Property~2), and some type of local embedding assumption (Property~3). Our cleaning results are about how given an edge coloring $C'$ of a pseudo-random blowup $G'$ of $G$, we can shrink the vertex sets of $G'$ slightly and define an auxiliary coloring $C$ of $E(G)$ so that we have this local embedding assumption in appropriate monochromatic subgraphs of $G'$.

The first cleaning result is a quantitative version of a lemma proved by 
Conlon, Nenadov, and Trujic \cite[Lemma 2.3]{CNT}. 

\begin{prp}\label{regularity cleaning}
    Fix $q,\Delta\ge 1$ and $p,\eta>0$. There exists a $\lambda = 
    \lambda(q,\Delta,p,\eta)>0$ so that the following holds. Let $G$ be a graph with maximum degree $\Delta$ and let $G'$ be a graph with a homomorphism $\phi:G'\to G$. Defining $X_v := \phi^{-1}(v)$ for each $v\in V(G)$, suppose that $|X_v|= s$ and $G'[X_u,X_v]$ is $(\lambda s,p)$-lower-regular for each $uv\in E(G)$. Then for any $q$-coloring $C'$ of $E(G')$, we can find subsets $X^*_v\subset X_v$ for $v\in V(G)$ and a $q$-coloring $C$ of $E(G)$, so that:
    \begin{enumerate}
        \item $|X_v^*|= s^*:=\lambda |X_v|$ for $v\in V(G)$;
        \item $G'_{C(uv)}[X_u^*,X_v^*]$ is $(\eta s^*, p/4q)$-lower-regular (where $G'_i$ denotes the subgraph of edges in $G'$ receiving color $i$).
    \end{enumerate}
Moreover, \[ \quad \lambda(q,\Delta,p,\eta) \ge (p/2q)^{14/(p/2q)^{14/(p/2q)^{\rddots {14/\eta}}}},\]
i.e., has a tower-type dependence with $\Delta+1$ occurences of $p/2q$.
\end{prp}

To get this quantitative bound, we recall a result of \cite{PRR}.

\begin{thm}\label{weak regularity lemma}
Consider a bipartite graph $G = (X,Y,E)$ with $|X| =|Y| := n$ and density $p := \frac{|E|}{|X||Y|}$, along with $\ep >0$. We can find $X'\subset X,Y'\subset Y$ of size $n':= \frac{1}{2}np^{12/\ep}$ so that $G[X',Y']$ is $(\ep n',p/2)$-lower-regular.
\end{thm}

This immediately gives the cleaning result for matchings, which we iterate.
\begin{lem}\label{matching clean}
    We have $\lambda(q,1,p,\eta)\ge \frac{1}{2} (p/2q)^{12/\eta} \geq (p/2q)^{13/\eta}$.
\end{lem}
\begin{proof}
    When $\Delta(G) =1$, its edges form a matching. For each edge $e=uv\in E(G)$, we can use lower-regularity and pigeonhole principle to find some color $c_e$ with $G'_{c_e}[X_u,X_v]$ having density at least $p/2q$. Then for each edge $e$ we assign $C(e):= c_e$ and apply Theorem~\ref{weak regularity lemma} (with $\epsilon=\eta$) to get sets $X_v^*$ (for vertices $v$ that are not covered by the matching we can take $X_v^*\subset X_v$ arbitrarily).
\end{proof}

\begin{proof}[Sketch of Proposition~\ref{regularity cleaning}]
    Set $\ep_0 := \eta$ and $\lambda_0 := (p/2q)^{13/\ep_0} \le \frac{1}{2}(p/2q)^{12/\ep_0}$, and for $i=1,\dots,\Delta$ set $\ep_i := \ep_{i-1}\lambda_{i-1}$ and $\lambda_i := (p/2q)^{13/\ep_i}$. For $t =0,\dots, \Delta$, write $\lambda^{\le t}:= \prod_{i=0}^t \lambda_i$; inductively we get that $\ep_t = \eta \lambda^{\le t-1}$. We shall show that one can take $ \lambda(p,\Delta,q,\eta) \ge \lambda^{\le \Delta}$.

    Indeed, to do this, we first apply Vizing's theorem to partition the edges of $G$ into $\Delta+1$ matchings $M_\Delta,M_{\Delta-1},\dots,M_{0}$ (since $G$ has maximum degree $\Delta$). Now fix some $q$-coloring of the edges of $G'$. We shall find our sets $X^*_v$ in $\Delta+1$ stages. 

    We initialize with $X_v^{(\Delta+1)} := X_v$ for each $v\in V(G)$. For $t=\Delta, \dots,0$, write $s_t:= (\prod_{i=t}^\Delta \lambda_t 
 )s$. The idea is that at stage $t=\Delta,\dots,0$, we can pass to a subset $X_v^{(t)} \subset X_v^{(t+1)}$ of size $\lambda_t |X_v^{(t+1)}| =s_t$, so that for each $e = uv \in M_t$ and some choice of $C(e)$, we have
$G'_{C(e)}[X_u^{(t)},X_v^{(t)}]$ is $(\ep_ts_t,p/4q)$-lower-regular. This is done by applying Lemma~\ref{matching clean} (with $\eta=\epsilon_t$).
    
    To finish, since $\ep_{t} = \eta \lambda^{\le t-1}$, we get that $\ep_ts_t = \eta  \prod_{i=0}^\Delta \lambda_i s= \eta s_0 $. Whence, at the end of the process, we will have a collection of subsets $X_v^{(0)}\subset X_v^{(1)}\subset \dots\subset  X_v^{(\Delta+1)}$, so that for any edge $e\in M_t$, $G'_{C(e)}[X_u^{(t)},X_v^{(t)}]$ (and thus $G'_{C(e)}[X_u^{(0)},X_v^{(0)}]$) is $(\eta s_0,p/4q)$-lower-regular and $X_v^{(0)}$ has size $s_0$. Taking $X_v^* := X_v^{(0)}$ for $v\in V(G)$ and the coloring $C$ will then complete the proof (with $s^* = s_0$).

    It is not hard to see that $\lambda^{\le \Delta} = \lambda^{\le \Delta-1} \lambda_\Delta \ge \ep_\Delta \lambda_\Delta$. Using that $\ep_t\ge (1/2)^{1/\ep_t} >(p/2q)^{1/\ep_t}$, one then gets the tower-type bound by recursively noting that $\ep_t\lambda_t \ge (p/2q)^{1/\ep_t}\cdot (p/2q)^{13/\ep_t}=(p/2q)^{14/\ep_t} = (p/2q)^{14/(\ep_{t-1}\lambda_{t-1})}$.
\end{proof}

We now move on to our more efficient cleaning process, which we prove for bipartite graphs.

\begin{prp}\label{DRC cleaning}
    Let $G=(A,B,E)$ be a bipartite graph with maximum degree $\Delta$. Let $G'$ be a graph with a homomorphism $\phi:V(G')\to V(G)$, and define $X_a := \phi^{-1}(a)$ for $a\in A$ and $Y_b := \phi^{-1}(b)$ for $b\in B$. For integers $r,q,L\ge 2$, $p\in (0,1)$ and all 
    $a\in A, b\in B$, suppose that $|Y_b| = s_0 \ge L(2q/p)^{5\Delta^2 r}$ and $|X_a| = s\ge 4\left(2q/p\right)^{5\Delta} \Delta s_0$. Also suppose that $G'[X_a,Y_b]$ is $(L,p)$-regular for $ab\in E(G)$. Then given any $q$-edge-coloring $\chi'$ of $G'$, we can produce an $q$-edge-coloring $\chi$ of $G$ and subsets $Y_b^*\subset Y_b$ for $b\in B$ with $|Y_b^*|\ge \left(\frac{p}{2q}\right)^{5\Delta^2 r} |Y_b|$, so that for all choices of $a\in A$, $b_1,\dots, b_r\in N_G(a)$, and $y_i\in Y_{b_i}^*$ for $i\in [r]$, we have that \[|\{x\in X_a: \chi'(xy_i) = \chi(ab_i) \text{ for each }i=1,\dots,r\}|\ge \sqrt{|X_a|/2q^\Delta}.\]
    \end{prp}

  \noindent
For our applications of finding $w$-blowups of $H$, it is important that in this statement $b_1,\dots,b_r$ might not all be distinct.

    While Proposition~\ref{regularity cleaning} worked by iterating upon matchings (cf. Lemma~\ref{matching clean}), we now shall instead iterate upon stars.
    \begin{lem}\label{star cleaning}
    Consider sets $X$ and $Y_1',\dots,Y_\ell'$ in a graph $G'$ so that all $|Y_i'| \ge L$, $|X|\ge 2\ell L$ and $|X|\ge (2q^\ell)(2(\frac{2q}{p})^{4\ell}) \sum_i |Y_i'|$. Also suppose that all $G'[X,Y_i']$ are $(L,p)$-regular. Then given any $q$-edge-coloring $\chi'$ of $G'$, we can find for each $i\in [\ell]$, a choice of $c_i\in [q]$ and $Y_i''\subset Y_i'$ with $|Y_i''| \ge \frac{1}{2}(\frac{p}{2q})^{4r\ell}|Y_i'|$ so that for all choices of $i_1,\dots, i_r\in [\ell]$, and $y_j\in Y_{i_j}''$ for $j\in [r]$, we have that \[|\{x\in X: \chi'(xy_j) = c_{i_j} \text{ for each }j=1,\dots,r\}|\ge \sqrt{|X|/2q^\ell}.\] 
\end{lem}
    
    \begin{proof}[Proof of Proposition~\ref{DRC cleaning} assuming Lemma~\ref{star cleaning}]
        Enumerate the vertices of $A$ as $a_1,\dots, a_{|A|}$ in some arbitrary order. For $t=0,\dots, |A|$, let $G_t:= G[\{a_1,\dots,a_t\},B]$ denote the induced subgraph between the first $t$ vertices in our ordering of $A$ and $B$. Set $\delta := \left(\frac{p}{2q}\right)^{5\Delta r}<\frac{1}{2}\left(\frac{p}{2q}\right)^{4\Delta r}$.

        For $t=0,\dots, |A|$,
        we will inductively construct a $q$-coloring $\chi_t$ of $E(G_t)$ and subsets $Y'_{b,t}\subset Y_b$ for $b\in B$ satisfying the following properties. All $|Y'_{b,t}| \ge \delta^{d_{G_t}(b)} \cdot s_0$ and
for any $a\in \{a_1,\dots,a_t\}$, $b_1,\dots,b_r\in N(a)$, and choices $y_i\in Y'_{b_i,t}$ for $1 \leq i\leq r$, we have that  
\begin{eqnarray}
\label{eq4.3}
|\{x\in X_a: \chi'(xy_i) = \chi_t(ab_i) \text{ for each }i=1,\dots,r\}|\ge \sqrt{|X_a|/2q^\Delta}.
\end{eqnarray}
    Clearly, taking $t=|A|$ will give our result, since $d_{G_{|A|}}(b) = d_G(b)\le \Delta$ for all $b\in B$.
    
    At time $t= 0$, take $Y'_{b,0} := Y_b$ for all $b\in B$. 
    Then at time $t>0$, assuming the conditions are satisfied, do the following. We shall find subsets $Y''_b\subset Y'_{b,t-1}$ with $|Y''_b|\ge \frac{1}{2} (p/2q)^{4\Delta r}|Y'_{b,t-1}|\ge \delta |Y'_{b,t-1}|$ and colors $c_b\in [q]$ for $b\in N(a_t)$, so that for any choice of $b_1,\dots,b_r\in N(a_t)$, and $y_i\in Y_{b_i}''$ we have that
    \[|\{x\in X_{a_t}: \chi'(xy_i) = c_{b_i} \text{ for each }i=1,\dots,r\}|\ge \sqrt{|X_a|/2q^\Delta}.\]
    Assigning $Y'_{b,t} := Y_b''$ for $b\in N(a_t)$ and $Y'_{b,t} = Y'_{b,t-1}$ otherwise, and assigning $\chi_t(a_tb) = c_b$ for $b\in N(a_t)$ (and $\chi_t(e)=\chi_{t-1}(e)$ otherwise) shall allow us to continue the induction. 
        
    Indeed, one easily checks that the bound on the size of the sets $Y'_{b,t}$ stay satisfied by induction. Meanwhile, by definition of the $Y_b''$ and $c_b$, the inequality Eq. \ref{eq4.3} shall be satisfied when $a=a_t$. Lastly, for $a= a_{t'}$ for $t'<t$, the inequality Eq.~\ref{eq4.3} follows from the inductive hypothesis, since shrinking the sets $Y'_{b,t-1}$ cannot invalidate this condition.

    It remains to find the $Y_b''$ and $c_b$. We will enumerate $N(a_t)$ as $b_1,\dots,b_\ell$ for some $\ell \le \Delta$. We then simply apply Lemma~\ref{star cleaning} (with $X= X_{a_t},Y_i':=Y'_{b_i, t-1}$). This can be done since $|Y'_{b, t-1}|\ge \delta^{\Delta} s_0=
    \left(\frac{p}{2q}\right)^{5\Delta^2 r} s_0\geq
    L$ for each $b$, and additionally we have $$\frac{|X_{a_t}|}{\sum_i |Y_{b_i, t-1}|}\ge \frac{s}{\Delta s_0} \ge 4\left(2q/p\right)^{5\Delta}\ge
    (2q^\Delta)(2(2q/p)^{4\Delta}).$$
\end{proof}

It remains to prove the star cleaning statement. We start with a simple lemma, which follows directly from the definition of regularity.
\begin{lem}\label{get mindegree}
Let $X$ and $Y_1,\dots,Y_\ell$ be subsets of $G'$ such that all $|Y_i|\ge L$ for $i\in [\ell]$, $|X|\ge 2\ell L$ and
all $G'[X,Y_i]$ are $(L,p)$-regular. Then for any $q$-edge-coloring of $G'$, we can find $X^* \subset X$ of size $|X^*|\ge \frac{|X|}{2q^\ell}$ and choices of $c_1,\dots,c_\ell\in [q]$ so that for $x\in X^*,i\in [\ell]$, $|N_{c_i}(x)\cap Y_i|\ge \frac{p}{2q}|Y_i|$.
\begin{proof}
    For $i=1,\dots,\ell$, let $\Tilde{X}_i$ be the set of $x\in X$ where $|N_{G'}(x)\cap Y_i|<\frac{p}{2}|Y_i|$. Using that $G'[X,Y_i]$ 
    is $(L,p)$-regular together with $|Y_i|\ge L$ we have that $\Tilde{X}_i\le L$ for each $i\in [\ell]$. Setting $X' := X\setminus \bigcup_{i=1}^t\Tilde{X}_i$, implies $|X'|\ge |X|- \ell L \geq \frac{1}{2}|X|$. 
 
    Next, for $x\in X'$ and $i\in [\ell]$, let $c_{i,x}$ be the color class that maximizes $|N_{c_{i,x}}(x)\cap Y_i|$. This implies that $|N_{c^*_i}(x)\cap Y_i|\ge (p/2q)|Y_i|$ for each $x\in X'',i\in [\ell]$. 
    By pigeonhole, there must be some tuple of colors $\vec{c}^*\in [q]^\ell$ so that the set 
    \[X'':= \{x\in X': c_{i,x}=c^*_i \text{ for all }i\in [\ell]\},\] satisfies $|X''|\ge |X'|/q^\ell$. 
    Thus, taking $X^* = X''$ and $c_i = c_i^*$ for $i\in [\ell]$ completes the proof.
\end{proof}

\end{lem}

\begin{proof}[Proof of Lemma~\ref{star cleaning}]
    Fix any $q$-coloring $\chi'$ of $E(G')$. Using Lemma~\ref{get mindegree}, we can pass to some $X^*\subset X$ of size $|X^*| \geq |X|/(2q^\ell)\geq 2(2q/p)^{4\ell}\sum_i |Y_i'|$, along with colors $c_1,\dots,c_\ell$ so that $|N_{c_i}(x)\cap Y_i'|\ge \frac{p}{2q}|Y'_i|$ for each $i\in [\ell]$ and $x\in X^*$.

    Setting $h:=4r$, we have that
    \[(\sum_i|Y_i'|)^r|X^*|^{-h/2} \le |X^*|^{-r} <\frac{1}{2} (p/2q)^{4r\ell}=\frac{1}{2} (p/2q)^{h\ell},\]
    where we used that $|X^*| \ge 2(2q/p)^{4\ell}$ for the last step.

    Let $\Gamma$ be the bipartite graph with vertex sets $X^*$ and $\overline{Y}:=\bigsqcup_i Y_i'$, where (for $x\in X^*,y_i\in Y_i'$) we have $x\sim_\Gamma y_i$ if $\chi'(xy_i) = c_i$. Applying Lemma~\ref{simultaneous DRC} to $\Gamma$ with $h=4r$, gives subsets $Y_i''\subset Y_i'$ of size $|Y_i''|\ge\frac{1}{2}(p/2q)^{h\ell} |Y_i'|$ which, by recalling the definition of $\Gamma$ and that $|X^*|\ge |X|/(2q^\ell)$, have the desired common neighborhood property.
\end{proof}

\section{Proof of main theorems}

\subsection{The two reductions}
\begin{proof}[Proof of Theorem~\ref{thm:reduction general}] Suppose we have been given some host graph $G$ so that $G\to (H)_q$. Recall $\Delta(H)\le k$ and $\Delta(G)\le \Delta$. Set $p := \frac{1}{100}, \rho := \frac{p}{4q}$ and let $\eta := \frac{(\rho/2)^k(1-2p)^\Delta}{\Delta+k}$. Take $\lambda := \lambda(q,\Delta,p,\eta)$ which is defined in Lemma~\ref{regularity cleaning}. 

By Proposition~\ref{regular}, we can find some bipartite graph $\Gamma$ with parts of size $s=C(\lambda\eta)^{-2}$ for some absolute constant $C$ which is $\big(\frac{48}{p}\ln (2s), p\big)$-regular. By choosing $C$ large enough we can make 
$\frac{48}{p}\ln (2s) \le s^{1/2} \leq \lambda \eta s$. Therefore $\Gamma$ is
$(\lambda \eta s, p)$-regular and thus obviously $(\lambda s,p)$-lower-regular. Take $s^*:= \lambda s$, and note $\Gamma$ is $(\eta s^*,p)$-regular.

Define $G'$ to have (disjoint) vertex sets $X_v$ of size $s$ for $v\in V(G)$. For every $uv\in E(G)$, we add a copy of $\Gamma$ between the sets $X_u$ and $X_v$ (i.e., $G'[X_u,X_v]$ is isomorphic to $\Gamma$). We do not add any other edges to $G'$, thus $G'$ is an $s$-blowup of $G$. Now consider any $q$-coloring $\chi'$ of $E(G')$. Since $G'[X_u,X_v] \cong\Gamma$ is $(\lambda s,p)$-lower-regular for $uv\in E(G)$, we may apply Lemma~\ref{regularity cleaning} to pass to subsets $X^*_v\subset X_v$ each of size $s^*=\lambda s$ and get an auxiliary $q$-coloring $\chi$ of $E(G)$ so that $G'_{\chi(uv)}[X_u^*,X_v^*]$ is $(\eta s^*,\rho)$-lower-regular for all $uv\in E(G)$ (recall that given a color $c\in [q]$, we write $G'_c$ to denote the $c$-monochromatic subgraph of $G'$ in the coloring $\chi'$). 

Since $G\to (H)_q$, we can find some monochromatic copy of $H$ inside $\chi$, say in color $c$. Let $f:V(H) \to V(G)$ denote a $c$-monochromatic copy of $H$. We now intend to apply our embedding procedure (Proposition~\ref{general greedy embedding}). We define $H^*$ to be the graph on vertex set $\bigcup_{v\in V(H)}X_{f(v)}^*$, with the edges $\bigcup_{uv\in E(H)} G'_c[X_{f(u)}^*,X_{f(v)}^*]$. 

Let $G^* = G'[\bigcup_v X^*_v]$ and without loss of generality assume that $f(v) =v$ for $v\in V(H)$. Then, we have $|X_v^*|=s^*$ for $v\in V(H)$,
and $G^*[X_u^*,X_v^*]$ is $(\eta s^*,p)$-regular for all $uv\in E(G)$. We also have $H^*[X_u^*,X_v^*]$ is $(\eta s^*,\rho)$-lower-regular for $uv\in E(H)$. Taking $L=L'=\eta s^*$ and using the definition of $\eta$, we can verify that 
$s^* \frac{(\rho/2)^k(1-2p)^\Delta}{\Delta+k}=\eta s^*=L$, i.e., 
Eq. \ref{s-star large} holds. Thus we may apply Proposition~\ref{general greedy embedding} to find an induced copy of $H$ in $H^*$. This copy corresponds to a monochromatic induced copy of $H$ within $G'$, completing the proof.
\end{proof}

We now establish our second reduction, whose proof is rather similar to the above arguments.
\begin{thm}\label{thm:reduction bipartite}
    There exists an absolute constant $C$ so that the following holds. 
    Let $H,G$ be bipartite graphs with $G\to (H)_q$. Suppose $\Delta(H)\le k, \Delta(G)\le \Delta$ and let $H'$ be some $w$-blowup of $H$.
    Then, for $s:= (\Delta q)^{Ck \Delta^2 w}$, there is an $s$-blowup $G'$ of $G$, so that $G'\to_{\text{ind}} (H')_q$. Consequently,
    \[r_{\text{ind}}(H';q)\le s|V(G)|,\]
    and
    \[\hat{r}_{\text{ind}}(H';q) \le s^2 e(G)\le s^2 \Delta |V(G)|.\]
\end{thm}
\begin{proof}
Let $G=(A,B,E)$ be a bipartite graph so that $G\to (H)_q$ and $\Delta(H)\le k,\Delta(G) \le \Delta$. Note that we must have $k\le \Delta$, since $H$ is a subgraph of $G$.
Set $p:= \frac{1}{4\Delta}$, and let $s:= (\frac{2q}{p})^{C\Delta^2k w}$. In the rest of the proof we assume that the constant $C$ is sufficiently large 
so that all the inequalities which we will use are satisfied. Fix $s_0 := s^{1/3}$ and use Corollary~\ref{regular}, to find some $(L,p)$-regular graph $\Gamma$ with parts of size $s,s_0$
and $L=\frac{48}{p} \ln(2s)\le 200 \Delta \ln (2s) \leq s^{1/9}=s_0^{1/3}$. Form a graph $G'$ by fixing a vertex set $X_a$ of size $s$ for all $a\in A$, and $Y_b$ of size $s_0$ for all $b\in B$. For $e=(a,b)\in E(G)$, add a copy of $\Gamma$ between $X_a,Y_b$. 

We claim that the graph $G'$ satisfies the assertion of the theorem.
Suppose we are given a $q$-edge-coloring $\chi'$ of $G'$. 
Use Proposition~\ref{DRC cleaning} with $r:= kw$ and $s, s_0$ defined above. 
Note that since the constant $C$ is large the conditions of this proposition are satisfied. Indeed $L(2q/p)^{5\Delta^2 r} \leq s_0^{1/3} (2q/p)^{5\Delta^2 kw}\leq s_0=|Y_b|$ and $4\left(2q/p\right)^{5\Delta} \Delta s_0 \leq (2q/p)^{6\Delta^2 kw} s^{1/3} \leq s=|X_a|$ for all $b \in B$ and $a\in A$.
Therefore, by Proposition~\ref{DRC cleaning}, there is some $q$-edge-coloring $\chi$ of $G$ and subsets $Y^*_b\subset Y_b$ of size 
$\left(\frac{p}{2q}\right)^{5\Delta^2 kw} s_0\ge ws_0^{1/2}$, so that for any $a\in A$, $b_1,\dots,b_{kw}\in N(a)$ and $y_i\in Y_{b_i}$ for $i\in [kw]$:
    \[|\{x\in X_a: \chi(xy_i) = \chi'(ab_i)\text{ for all }i\in [kw]\}|\geq  \sqrt{|X_a|/2q^\Delta} \geq s^{1/2}/q^\Delta\geq s^{1/4}.\] 

Since $G\to (H)_q$ we can find a monochromatic copy of $H$ inside the $q$-edge-coloring $\chi$ of $G$, say in color $c$. Consider the induced $c$-monochromatic 
subgraph $H^*$ of $G'_c$ whose vertex set is the union of sets $X_a$, $Y^*_b$ for $a, b$ that correspond to the vertices of the monochromatic copy of $H$ which we found in $\chi$. Let $s^*:=s_0^{1/2}$. We have that $|Y^*_b| \geq ws^*$, $G'[X_a,Y^*_b]$ is $(L,p)$-regular for any 
edge $ab \in E(G)$ and for any $a\in A$, $b_1,\dots,b_{kw}\in N_H(a)$ and $y_i\in Y^*_{b_i}$ we have
$\big|X_a \cap \bigcap_{i=1}^{wk} N_{H^*}(y_i)\big| \geq s^{1/4}$. Using that $s=(\frac{2q}{p})^{C\Delta^2k w}$, $s_0=s^{1/3}$ and $L \leq s^{1/9}=s_0^{1/3}$, it is easy to see that $(1-2p )^{\Delta w} s^{1/4}\geq 2^{-\Delta w} s^{1/4} \geq s^{1/5} \geq wL$ and that $w\Delta\frac{L}{s^*} \leq w\Delta s_0^{-1/6}<s_0^{-1/7}<\frac{1}{e(w\Delta^2+1)}$.
Therefore we can apply Lemma~\ref{DRC embedding} to find an induced monochromatic copy of $H'$ inside $G'$.
\end{proof}

\subsection{Obtaining Theorems~\texorpdfstring{\ref{thm:bonus}}{2} and \texorpdfstring{\ref{thm:main}}{1}.}
    \label{final deductions}

\begin{defn}
    Given graph $G,H$ and $\gamma >0$, we say $G\to_\gamma H$ if for every $\Tilde{G}\subset G$ with $e(\Tilde{G})\ge \gamma e(G)$, we have that $\Tilde{G}$ contains a copy of $H$.
\end{defn}
By using Friedman-Pippenger embedding techniques, Haxell and Kohayakawa proved (cf. \cite[Lemma~6 and the proof of Theorem~9]{HK}):

\begin{thm}\label{random Ramsey result}
There exists an absolute constant $C$ so that the following holds. 
Let $T$ be an $n$-vertex tree with maximum degree $k$ and let $\gamma\in (0,1)$. If
$N:= C\gamma^{-2} n$ and $p := \frac{C\gamma^{-2}k}{N}$, then for a random bipartite graph $G\sim G(N,N,p)$, we have that $\P(G\to_\gamma H)\ge 1/2$.
\end{thm}
\noindent Note that the maximum degree of such random graph is not bounded. But the results of \cite{HK} are more general and say that any $G$ with appropriate expansion properties will have $G\to_\gamma T$. In particular, we can take $G$ to be a random bipartite $O(\gamma^{-2}k)$-regular graph on $O(\gamma^{-2}n)$ vertices. 
By taking $\gamma = 1/q$ we have that $G\to (T)_q$ and has maximum degree $O(q^2k)$.
Whence Theorem~\ref{thm:main} follows from Theorem~\ref{thm:reduction bipartite}. Alternatively, one can use explicit bounded degree expanders, as described in \cite[Section~5]{HK}, to make $G$ (and the proof of Theorem~\ref{thm:main}) constructive.

Next, to prove Theorem~\ref{thm:main}, we recall a very recent result \cite[Corollary~2]{BKMMMMP} for size-Ramsey of graphs with bounded degree and bounded tree-width. We actually use a slightly more precise statement, \cite[Theorem~3.1]{DKMPS}, which observes the construction from \cite{BKMMMMP} has bounded maximum degree, rather than just linearly many edges.
\begin{thm}\label{blackbox treewidth} For every $k,w,q$ there exists a constant $D = D_{k,w,q}$ so that the following holds. 
Let $H$ be an $n$-vertex graph with maximum degree $\le k$ and tree-width $\le w$. Then there is some $G$ with at most $Dn$ vertices and $\Delta(G) \le D$, so that $G\to (H)_q$.
\end{thm}
\noindent
Theorem~\ref{thm:main} now can be obtained using the above statement followed by Theorem~\ref{thm:reduction general}.

\begin{rmk}
    The statements of Theorems~\ref{thm:reduction general} and \ref{thm:reduction bipartite} require the host graph $G$ to
    have linearly many vertices and bounded maximum degree. A priori, this may sound stronger than having a host graph $G$ with linearly many edges, but these properties are almost equivalent due to a simple general reduction. Indeed, suppose $H$ is an $n$-vertex graph with maximum degree $k$ and no isolated vertices, and $G_0$ is a graph with $Dn$ edges so that $G_0\to (H)_{q+1}$. Then letting $G\subset G_0$ be the graph induced by the non-isolated vertices in $G_0$ with degree at most $4kD$, one can check that $G\to (H)_q,|V(G)|\le 2Dn$, and obviously $\Delta(G)\le 4kD$.
\end{rmk}

\section{Conclusion}

As noted in the introduction, there exist $n$-vertex trees $T$ with maximum degree $k$ so that $\hat{r}_{\text{ind}}(T)\ge \hat{r}(T)= \Omega(nk)$. But for any tree $T$, we know that $r(T)= O(n)$, so one might wonder if we could have the linear bound $r_{\text{ind}}(T) = O(n)$ (where the constant does not depend on the maximum degree $k$). 
This is refuted by a result of Fox and Sudakov \cite[Theorem~1.7]{FS}, which says that for every constant $C$, there exist $n$-vertex trees $T$ with $r_{\text{ind}}(T)\ge Cn$ for all sufficiently large values of $n$.
However, the argument of \cite{FS} uses trees with maximum degree growing with $n$ (they have $k = \Omega(n)$). We are not aware of a counterexample to the claim that $r_{\text{ind}}(T) \le  O(n) +O_k(1)$; which in words would mean there is some absolute constant $C$, so that for fixed $k$ there are only finitely many $n$-vertex trees $T$ with maximum degree $k$ violating $r_{\text{ind}}(T)\le Cn$.

When $H$ is bipartite, one could naturally ask for a density version of our results. Recall that $G\to_\ep H$ if for every $\Tilde{G}\subset G$ with $e(\Tilde{G})\ge \ep e(G)$, we have that $H$ is a subgraph of $\Tilde{G}$. The result of Beck \cite{beck} proved that for each $\ep>0$, there is a $G$ with $O_\ep(n)$ vertices so that $G\to_\ep P_n$; likewise the bounds of \cite{FP} and \cite{HK} for arbitrary bounded degree trees also were density results. We now write $G\to_{\text{ind},\ep}H$ if for every $\Tilde{G}\subset G$ with $e(\Tilde{G})\ge \ep e(G)$, there is some set of vertices $S$ so that $\Tilde{G}[S] \cong H \cong G[S]$. 

A minor generalization of our proof of Theorem~\ref{thm:reduction general} gives the following. 
\begin{thm}
\label{density}
Given $\Delta,k\ge 1$ and $\ep'>\ep>0$, there exists some $s$ so that the following holds. 
Let $G,H$ be graphs so that $\Delta(G)\le \Delta,\Delta(H)\le k$ and $G\to_\ep H$. Then there exists an $s$-blowup $G'$ of $G$ so that $G'\to_{\text{ind},\ep'} H$.
\end{thm}

\noindent To construct $G'$, we again replace each $v\in V(G)$ by a set $X_v$ of $s$ vertices and put a pseudorandom graph $\Gamma$ inside $G'[X_u,X_v]$ for each $uv\in E(G)$. And like before, we shall ultimately use Proposition~\ref{general greedy embedding} to embed $H$. But for the cleaning, we must use the full strength of graph regularity, rather than the weaker form stated in Theorem~\ref{weak regularity lemma}. Here is the necessary cleaning result. 
\begin{prp} Fix $\Delta\ge 1$, and some $\eta,c>0$. There exists a $\lambda = \lambda(\Delta,\eta,c)$ so that the following holds.
Let $G$ be a graph with maximum degree $\Delta$ and let $\Tilde{G}$ be an $s$-blowup of $G$ with $(X_v)_{v\in V(G)}$ being sets of size $s$ that were mapped to the vertices of $G$. Then there is randomized procedure to choose subsets $X_v^*\subset X_v$ of size $s^*:= \lambda s$, so that, for each $e=uv\in E(G)$, we have 
\[\P(\Tilde{G}[X_u^*,X_v^*]\text{ is $(\eta s^*,p_{uv})$-regular for some }p_{uv}\ge c)\ge \frac{1}{s^2}e(\Tilde{G}[X_u,X_v]) -c-\eta. \]
In particular, if $e(\Tilde{G}) = \ep's^2 e(G)$, then there is some outcome of $(X_u^*)_{v\in V(G)}$ where there are $\ge (\ep'-c-\eta)e(G)$ edges $e=uv\in E(G)$ where $\Tilde{G}[X_u^*,X_v^*]$ is $(\eta s^*,c)$-lower-regular.
\end{prp}

\noindent
This implies that, if $G'$ is a pseudorandom $s$-blowup with $p s^2 e(G)$ edges, and $\Tilde{G}\subset G'$ with $e(\Tilde{G})\ge \ep' e(G')$, then this proposition yields a ``regular'' $s^*$-blowup $G^*\subset \Tilde{G}$, where $G^*[X_u^*,X_v^*]$ is lower-regular for at least an $(p\ep'-c-\eta)$-fraction of edges $e=uv\in E(G)$. Assuming we took  $\ep'>\ep/p$ and $\eta,c$ sufficiently small, then by definition of $G$ there should be a copy of $H$ inside $G$ using only these regular edges. Letting $H^*\subset G^*$ be the appropriate subgraph associated with $H$, we can apply Proposition~\ref{general greedy embedding} to find an induced copy of $H$.

Here we sketch how to find the necessary $X_v^*$. We let $\delta_1,\dots,\delta_{\Delta+1},\ell_1,\dots,\ell_{\Delta+1}$ be appropriately chosen constants.
Firstly, $\delta_{\Delta+1}$ should equal $\eta$. Next, $\ell_i$ should be large with respect to $\delta_i,c$ so that we can apply the Regularity Lemma in order to get a ``$\delta_i$-regular partition'' into $\ell_i$ equal-sized parts. Lastly we require $\delta_{i-1}<\delta_i/\ell_i$ for $i>1$. Take $s := s^*\prod_{i=1}^{\Delta+1} \ell_i$ for some $s^*$ sufficiently large. Note that we have $\delta_i \frac{s}{\prod_{j\leq i}\ell_j} \leq \delta_{\Delta+1} s^*$. We run the following process.

Given some $\Tilde{G}\subset G'$, split $E(G)$ into matchings $M_1,\dots,M_{\Delta+1}$. For $v\in V(G)$, we write $X_{v,0}:= X_v$, and then begin phase $t=1$. 
In phase $t$, we have sets $(X_{v,t-1})_{v\in V(G)}$ of size $s_{t-1} := \frac{s}{\prod_{i<t} \ell_i}$.
For every edge $e=uv\in E(M_t)$, by applying Regularity Lemma to the graph $\Tilde{G}[X_{u,t-1},X_{v,t-1}]$
we can get for every $v \in G$ an equipartition $X_{v,t-1} = X_v^{(1)}\cup \dots \cup X_v^{(\ell_t)}$ into $\ell_t$ parts of size $s_t$, so that
for each $e=uv\in E(M_t)$, there are at most $\delta_t \ell_t^2$ ``bad choices'' of $i,j\in [\ell_t]^2$. Here we say that the choice is bad if writing $p_e^{(i,j,t)} := \frac{1}{s_t^2}\Tilde{G}[X_u^{(i)},X_v^{(j)}]$, we have $p_e^{(i,j,t)}\ge c$ but $\Tilde{G}[X_u^{(i)},X_v^{(j)}]$ is not $(\delta_t s_{t}, p_{e}^{(i,j,t)})$-regular. 
Then for each $v\in V(G)$, we randomly pick some $i_{v,t}\in [\ell_t]$ and set $X_{v,t} := X_{v,t-1}^{i_{v,t}}$.
Repeat this process until phase $\Delta+1$ completes. Afterwards, we return the vertex sets $(X_{v,\Delta+1})_{v\in V(G)}$ and also write $X_v^*:= X_{v,\Delta+1}$ for each $v\in V(G)$.

Given an edge $e\in E(M_t)$, we say that $e=uv$ is \textit{useful}, if in phase $t$, we picked indices $i:= i_{u,t}$ and $j:= i_{v,t}$ from $[\ell_t]$, so that $p_e^{(i,j,t)}\ge c$ and $\Tilde{G}[X_{u,t},X_{v,t}]$ is $(\delta_ts_t,p_e^{(i,j,t)})$-regular (meaning the choice $i,j$ was not ``bad'', in the sense described above). Note that if $e$ is useful, then, since  $\delta_t s_t \leq \delta_{\Delta+1} s^*$, we will have that $\Tilde{G}[X_u^*,X_v^*]$ is $(\eta s^*,p_e^{(i,j,t)})$-regular. So, it will suffice to show that for each $e=uv\in E(G)$, that $\P(e\text{ is useful})\ge \frac{1}{s^2}e(\Tilde{G}[X_u,X_v])-c-\eta $. Note that \[\P(e\text{ is useful}) = \P(e(\Tilde{G}[X_{u,t},X_{v,t}])\ge c s_t^2)-\P(e \text{ is bad}).\]
By construction of the $\delta_t$-regular partitions of $X_{u,t-1}$ and $X_{v,t-1}$, we have that $\P(e\text{ is bad})\le \delta_t\le \eta$. Hence we just need to prove that $\P(e(\Tilde{G}[X_{u,t},X_{v,t}])\ge cs_t^2)\ge \frac{1}{s^2}e(\Tilde{G}[X_u,X_v])-c$.

To prove this last bound, consider the sequence of random variables $Z_i := \frac{1}{s^2_i}e(\Tilde{G}[X_{u,i},X_{v,i}])$. From definitions, it is easy to see that
$\E[Z_i|Z_{i-1}]=Z_{i-1}$ and $\E[Z_0]=p$. Therefore $\E[Z_i]=p$ for all $i\geq 0$. Since $Z_i$ is $[0,1]$-valued, $p=\E[Z_t]\le c+\P(Z_t\ge c)$.
Rearranging then gives the result.

\vspace{0.25cm}
\noindent
{\bf Acknowledgement.} We would like to thank Rajko Nenadov for asking us about density version of our results, and suggesting that some density variant of Proposition~\ref{regularity cleaning} should be possible. This led us to the proof of Theorem \ref{density}. We also thank Nemanja Dragani\'c and Rajko Nenadov for offering some helpful comments on a preliminary version of this manuscript.


\begin{thebibliography}{}
\bibitem{AS} N. Alon and J. Spencer, \textit{The probabilistic method,} fourth ed., Wiley Series in Discrete Mathematics and
Optimization, John Wiley \& Sons, Inc., Hoboken, NJ, 2016.
\bibitem{beck}J. Beck, \textit{On size Ramsey number of paths, trees, and circuits. I,} in \textit{J. Graph Theory} \textbf{7} (1983), p. 115–129.
\bibitem{beck2} J. Beck, \textit{On size Ramsey number of paths, trees, and circuits. II,} in \textit{Mathematics of Ramsey theory, Algorithms Combin.,} Vol. 5, Springer, Berlin, 1990, p. 34–45.
\bibitem{BKMMMMP} S. Berger, Y. Kohayakawa, G. S. Maesaka, T. Martins, W. Mendon\c{c}a,
G. O. Mota, and O. Parczyk, \textit{The size-Ramsey number of powers of bounded degree
trees,} in \textit{Journal of the London Mathematical Society} \textbf{103} (2021), p. 1314–1332.
\bibitem{BDS}  D. Brada\v{c}, N. Dragani\'c, and B. Sudakov, \textit{Effective bounds for induced size-Ramsey numbers of cycles,} arXiv:2301.10160.
\bibitem{CFS1} D. Conlon, J. Fox, and B. Sudakov, \textit{On two problems in graph Ramsey theory,} in \textit{Combinatorica} \textbf{32} (2012), p. 513–535.
\bibitem{CFS} D. Conlon, J. Fox, and B. Sudakov, \textit{Recent developments in graph Ramsey theory,} in \textit{Surveys in
Combinatorics 2015,} Cambridge University Press, (2015) p. 49–118.

\bibitem{CNT} D. Conlon, R. Nenadov, and M. Trujic, \textit{On the size-Ramsey number of grids,} in \textit{Combinatorics, Probability and Computing} \textbf{32} (2023), p. 874-880.
\bibitem{CJKMMRR} D. Clemens, M. Jenssen, Y. Kohayakawa, N. Morrison, G. O. Mota, D. Reding, and B. Roberts, \textit{The size-Ramsey number of powers of paths,} in \textit{J. Graph Theory} \textbf{91} (2019), p. 290–299. 
\bibitem{deuber} W. Deuber, \textit{A generalization of Ramsey’s theorem,} In: Infinite and Finite Sets, Vol. 1, Colloquia Mathematica
Societatis János Bolyai, Vol. 10, North-Holland, Amsterdam/London, 1975, p. 323–332.
\bibitem{DGK} N. Dragani\'c, S. Glock, and M. Krivelevich, \textit{Short proofs of long induced paths,} in \textit{Combinatorics, Probability and Computing} \textbf{31} (2022), p. 870-878.
\bibitem{DKMPS} N. Dragani\'c, M. Kaufmann, D. Munh\'a Correia, K. Petrova, and R. Steiner, \textit{Size-Ramsey numbers of structurally sparse graphs,} arXiv:2307.12028.
\bibitem{EFRS}P. Erd\H{o}s, R. J. Faudree, C. C. Rousseau, and R. H. Schelp, \textit{The size Ramsey number,} in \textit{Periodica
Mathematica Hungarica} \textbf{9} (1978), p. 145–161.
\bibitem{EHP} P. Erd\H{o}s, A. Hajnal, and L. P\'osa, \textit{Strong embeddings of graphs into colored graphs,} In: Infinite and
Finite Sets, Vol. 1, Colloquia Mathematica Societatis János Bolyai, Vol. 10, North-Holland, Amsterdam/London,
1975, p. 585–595.
\bibitem{erdos1}P. Erd\H{o}s, \textit{Problems and results on finite and infinite graphs,} in \textit{Recent Advances in Graph Theory. Proceedings
of the Second Czechoslovak Symposium,} pages 183–192. Academia, Prague, 1975.
\bibitem{erdos} P. Erd\H{o}s, \textit{On the combinatorial problems which I would most like to see solved,} in \textit{Combinatorica} \textbf{1} (1981), p. 25–42.
\bibitem{FS} J. Fox and B. Sudakov, \textit{Induced Ramsey-type theorems,} in \textit{Advances in Mathematics} \textbf{219} (2008), p. 1771-1800.
\bibitem{FS2} J. Fox and B. Sudakov, \textit{Dependent random choice,} in \textit{Random Structures \& Algorithms} \textbf{38} (2011),
p. 1–32.
\bibitem{FP} J. Friedman and N. Pippenger, \textit{Expanding graphs contain all small trees,} in \textit{Combinatorica} \textbf{7} (1987), p. 71-76.

\bibitem{GH} A. Gir\~ao and E. Hurley, \textit{Embedding induced trees in sparse expanding graphs,} manuscript. 

\bibitem{GRS} R. L. Graham, B. L. Rothschild, and J. H. Spencer, Ramsey theory, 2nd edition, Wiley, 1990. 

\bibitem{HJKMR} J. Han, M. Jenssen, Y. Kohayakawa, G. O. Mota, and B. Roberts, \textit{The multicolour size-Ramsey number of powers of paths,} in \textit{J. Combin. Theory Ser. B} \textbf{145} (2020), p. 359–375. 
\bibitem{HK} P. Haxell and Y. Kohayakawa, \textit{The size-Ramsey number of trees,} in \textit{Israel Journal of
Mathematics} \textbf{89} (1995), p. 261–274.

\bibitem{HKL} P. Haxell, Y. Kohayakawa, and T. {\L}uczak,
\textit{The induced size-Ramsey number of cycles,} in \textit{Combinatorics, Probability and Computing} \textbf{4} (1995), p. 217–239.

\bibitem{JMW} T. Jiang, K. Milans, and D. B. West, \textit{Degree Ramsey numbers for cycles and blowups of trees,} in \textit{Combinatorics, Probability and Computing} \textbf{21} (2012), p. 229–253.

\bibitem{KLWY} N. Kam{\v{c}}ev, A. Liebenau, D. R. Wood, and L. Yepremyan, \textit{The size Ramsey number of
graphs with bounded treewidth,} in \textit{SIAM J. Discrete Mathematics} \textbf{35} (2021), p. 281–293.

\bibitem{KS} M. Krivelevich and B. Sudakov, \textit{Psuedo-random graphs,} in \textit{More Sets, Graphs and Numbers,} Bolyai Society Mathematical Studies \textbf{15}, Springer, (2006), p. 199-262.

\bibitem{PRR} Y. Peng, V. R\"odl, and A. Ruci\'nski, \textit{Holes in graphs,} in \textit{Electronic Journal of Combinatorics} \textbf{9} (2002), 18 pp. 

\bibitem{rodl} V. R\"odl, \textit{The dimension of a graph and generalized Ramsey theorems,} Master’s thesis, Charles University,
1973.



\end{thebibliography}
\end{document}